\documentclass[11pt]{amsart}

\usepackage{amsmath, amssymb, amsfonts,enumerate}
\usepackage[all]{xy}
\usepackage{amscd}
\usepackage{comment}
\usepackage{mathtools}
\usepackage{tikz} 
\usepackage{tikz-cd} 
\tikzcdset{diagrams={nodes={inner sep=7.5pt}}}

\usepackage{url}
\usepackage{hyperref}

\newtheorem{theorem}{Theorem}[section]
\newtheorem{lemma}[theorem]{Lemma}
\newtheorem{corollary}[theorem]{Corollary}
\newtheorem{proposition}[theorem]{Proposition}
\newtheorem{theoremletter}{Theorem}

\newtheorem{corollaryletter}{Corollary}

 \theoremstyle{definition}
 \newtheorem{definition}[theorem]{Definition}

\newtheorem{question}{Question}

  \newtheorem*{example*}{Example}

\numberwithin{equation}{section}
 
\newcommand {\N}{\mathbb{N}} 
\newcommand {\Z}{\mathbb{Z}}

\newcommand{\PP}{\mathcal{P}}

\DeclareMathOperator{\CA}{CA}

\DeclareMathOperator{\Ker}{Ker}
\DeclareMathOperator{\End}{End}

\DeclareMathOperator{\Id}{Id}

\DeclareMathOperator{\Spec}{Spec}

\begin{document}
\title[Weakly surjunctive groups symbolic group varieties]{Weakly surjunctive groups and symbolic group varieties}   
\author[Xuan Kien Phung]{Xuan Kien Phung}
\email{phungxuankien1@gmail.com}
\subjclass[2010]{14A10, 14A15, 16S34, 20C07, 37B10, 68Q80}
\keywords{group ring, stable finiteness, sofic group, surjunctive group, algebraic group, cellular automata, symbolic group varieties, reversibility, invertibility} 
 
\begin{abstract}
In this paper, we introduce the classes of weakly surjunctive and linearly surjunctive  groups which include all sofic groups and more generally all surjunctive groups. We investigate various properties of such groups and establish in particular a reversibility and invertibility theorem for injective endomorphisms of symbolic group varieties over weakly surjunctive group universes with algebraic group alphabets in arbitrary characteristic. We also obtain  novel evidence related to Kaplansky's stable finiteness conjecture. 
\end{abstract}
\date{\today}
\maketitle
  
\setcounter{tocdepth}{1}
 
\section{Introduction} 
To state the main results, we fix some basic notions of symbolic dynamics. Fix a set $A$ called the \emph{alphabet}, and a group  $G$, the \emph{universe}.
A \emph{configuration} $c \in A^G$ is a map $c \colon G \to A$. 
The Bernoulli shift $G \times A^G \to A^G$ is defined by $(g,c) \mapsto g c$, 
where $(gc)(h) \coloneqq  c(g^{-1}h)$ for all $g,h \in G$ and $c \in A^G$. 
\par
Introduced by von Neumann \cite{neumann}, a \emph{cellular automaton} over  the group $G$ and the alphabet $A$ is a map
$\tau \colon A^G \to A^G$ admitting a finite \emph{memory set} $M \subset G$
and a \emph{local defining map} $\mu \colon A^M \to A$ such that 
\begin{equation*} 
\label{e:local-property}
(\tau(c))(g) = \mu((g^{-1} c )\vert_M)  \quad  \text{for all } c \in A^G \text{ and } g \in G.
\end{equation*} 
\par 
Now let $X$  be an algebraic group over a field $k$. Denote by  $X(k)$ the set of rational points of $X$. The ring  $CA_{algr}(G,X,k)$ of \emph{algebraic group  cellular automata} consists of cellular automata $\tau \colon X(k)^G \to X(k)^G$ which  admit a memory  $M\subset G$ and 
a local defining map $\mu \colon X(k)^M \to X(k)$ induced by some $k$-homomorphism of algebraic groups  
$f \colon X^M \to X$, i.e., $\mu=f\vert_{X(k)^M}$, 
where $X^M$ is the fibered product of copies of $X$ indexed by $M$. 
\par 
A cellular automaton $\tau \colon A^G \to A^G$ is \emph{reversible} if it is bijective and the inverse map $\tau^{-1} \colon A^G \to A^G$ is also a cellular automaton. It is well-known that if the alphabet $A$ is finite then every bijective cellular automaton $\tau \colon A^G \to A^G$ is reversible (cf., e.g. \cite[Theorem~1.10.2]{ca-and-groups-springer}). 
\par 
In \cite[Theorem~1.1]{phung-2020}, the author proved the following surjunctivity and invertibility result for endomorphisms of symbolic group varieties:  

\begin{theorem}
\label{t:phung-2020-sofic-surjunctive}
Let $G$ be a sofic group. 
Let $X$ be a connected algebraic group over an algebraically closed field $k$ of characteristic zero. 
Suppose that $\tau \in  CA_{algr}(G,X,k)$ is injective. 
Then $\tau$ reversible and $\tau^{-1} \in CA_{algr}(G,X,k)$.  
\end{theorem}
\par 
Here, the wide class of sofic groups was first introduced by Gromov \cite{gromov-esav} as a common generalization of residually finite groups and amenable groups. 
\par 
Let $G$ and $A$ be groups. Then $A^G$ admits a natural  group structure with pointwise operations: for $x,y \in A^G$, the product $z=xy$ is given by $z(g)=x(g)y(g)$ for all $g \in G$. We call a map $\tau \colon A^G \to A^G$ a \emph{group cellular automaton} if $\tau$ is a cellular automaton and that $\tau(xy)=\tau(x)\tau(y)$ for all configurations $x,y \in A^G$. Equivalently, it is clear that a cellular automaton $\tau \colon A^G \to A^G$ is a group cellular automaton if $\tau$ admits for some finite memory set $M \subset G$ a local defining map $\mu \colon A^M \to A$ which is a homomorphism of groups. 
\par 
In this paper, we obtain a generalization of Theorem~\ref{t:phung-2020-sofic-surjunctive} where we can eliminate the connectedness hypothesis on the alphabet  and the universe is now only required to be a \emph{weakly surjunctive group} defined as follows. 

\begin{definition} 
We call a group $G$ a \emph{weakly surjunctive group} if for every finite group alphabet $A$, all injective group  cellular automata $\tau \colon A^G \to A^G$ are  surjective. A group $G$ is \emph{surjunctive} if for every finite  alphabet $A$, all injective  cellular automata $\tau \colon A^G \to A^G$ are  surjective.   
\end{definition} 
\par  
Surjunctive groups were introduced by Gottschalk in \cite{gottschalk}.  
It is clear from the definition that every surjunctive group is weakly surjunctive. Note that every sofic group is surjunctive by the Gromov-Weiss theorem (\cite{gromov-esav}, \cite{weiss-sgds}). We show in Lemma~\ref{l:locally-surjunctive} that a group $G$ is weakly surjunctive if and only if so is every finitely generated subgroup of $G$. 
\par The main result of the paper is the following: 

\begin{theoremletter}
\label{t:surjunctive-invertible}
Let $G$ be a weakly surjunctive group and let   $X$ be an algebraic group over an uncountable algebraically closed field $k$. 
Suppose that $\tau~\in~ CA_{algr}(G,X,k)$ is injective. Then $\tau$ is reversible and if $k$ has characteristic zero, one has   $\tau^{-1} \in CA_{algr}(G,X,k)$. 
\end{theoremletter}
\par 
As an immediate  application of the notion of weakly surjunctive groups, we mention here the following direct finiteness property for endomorphisms of symbolic group varieties: 

\begin{corollaryletter}
\label{t:direct-weak-surjunctive} 
Let $G$ be a weakly surjunctive group and let   $X$ be an algebraic group over an algebraically closed field $k$. 
Suppose that $\sigma, \tau \in CA_{algr}(G,X,k)$ satisfy 
 $\sigma \circ \tau= \Id$.  Then one also has $\tau  \circ \sigma = \Id$. 
\end{corollaryletter}
\par 
When $G$ is required to be surjunctive  while we only need to suppose that $\tau, \sigma \in CA_{alg}(G,X,k)$ (see, e.g.  \cite{cscp-alg-ca}, \cite{cscp-alg-goe}, \cite{cscp-invariant-ca-alg}, \cite{phung-geometric}), the statement of Corollary~\ref{t:direct-weak-surjunctive} was proven in \cite{phung-geometric}. In fact, a proof of Corollary~\ref{t:direct-weak-surjunctive} can be easily obtained by a straightforward modification of the proof of  \cite[Theorem~C]{phung-geometric}. In Section~\ref{s:proof-direct-finite}, we explain the details and give another proof using Theorem~\ref{t:surjunctive-invertible}. 
\par  
We present briefly an application of Corollary~\ref{t:direct-weak-surjunctive} on \emph{Kaplansky's stable finiteness conjecture}. A ring $R$ \emph{stably finite} every one-sided invertible square matrix with coefficient in $R$ is also two-sided invertible. Then  {Kaplansky's stable finiteness conjectures} \cite{kap} asserts that for every group $G$ and every field $k$, the group ring $k[G]$ is stably finite. 
\par 
We obtain the following slight  generalization of \cite[Theorem~D]{phung-geometric} where the group $G$ is required to be surjunctive. 

\begin{corollaryletter}
\label{c:intro-1} 
Let $G$ be a weakly surjunctive group and let $R$ be the endomorphism ring of a commutative algebraic group over an algebraically closed field. Then the group ring $R[G]$ is stably finite. 
\end{corollaryletter}
\par 

\begin{proof}
(See \cite[Section~5.1]{phung-geometric})
Let $k$ be an algebraically closed field and let $Y$ be a commutative $k$-algebraic group. Denote by $R=\End_k(Y)$ the endomorphism ring of $k$-endomorphisms of $Y$. Fix $n \in \N$ and let $X=Y^n$. Then  $\mathrm{Mat}_n(R)=\End_k(X)$  by \cite[Lemma~9.5]{phung-2020} and  we have an  isomorphism of rings $\End_k(X)[G] \simeq CA_{algr}(G,X,k)$ by  \cite[Proposition~9.3]{phung-2020}.  Corollary~\ref{t:direct-weak-surjunctive}  implies that $CA_{algr}(G,X,k)$ is directly finite. It follows that  $\mathrm{Mat}_n(R)[G]=\End_K(X)[G]$ is also a directly finite ring. Hence, we conclude from the ring isomorphism  $\mathrm{Mat}_n(R[G])\simeq \mathrm{Mat}_n(R)[G]$ given by \cite[Lemma~9.4]{phung-2020} that  $\mathrm{Mat}_n(R[G])$ is a directly finite ring. 
\end{proof}
\par
To this point, we find that it is natural to ask the following question: 
\begin{question}
Does there exist a non weakly surjunctive group? 
\end{question}
\par 
The paper is organized as follows. We collect in Section~\ref{s:preliminary} basic lemmata in algebraic group theory. In  Section~\ref{s:weakly-surjunctive-groups},  we investigate basic properties of the class of weakly surjunctive groups. Then Section~\ref{s:proof-main} presents a short proof of Theorem~\ref{t:surjunctive-invertible} with ingredients from algebraic group theory and some ideas from the papers  \cite{phung-geometric} and  \cite{gromov-esav}. Note that in the case of sofic groups, the proof of Theorem~\ref{t:phung-2020-sofic-surjunctive} given in \cite{phung-2020} is more direct and quite different in nature from our proof of Theorem~\ref{t:surjunctive-invertible}. 
Section~\ref{s:proof-direct-finite} contains two different proofs of Corollary~\ref{t:direct-weak-surjunctive}. 
We define the class of \emph{linearly surjunctive groups} in Section~\ref{s:generalizations} which includes all weakly surjunctive groups and shows in particular that a group satisfies Kaplansky's stable finiteness conjecture if and only it is a linearly surjunctive group (Theorem~\ref{t:L-surjunctive-linear}). Finally, \emph{weakly dual surjunctive groups}, which are generalizations of \emph{dual surjunctive groups} \cite{kari-post-surjective}, will be introduced and investigated in  Section~\ref{s:weakly-dual-surjunctive}. In particular, we extend  Corollary~\ref{t:direct-weak-surjunctive} and Corollary~\ref{c:intro-1} to cover the case of weakly dual surjunctive group universes (see Theorem~\ref{t:stable-weakly-dual} and Corollary~\ref{c:stable-weakly-dual} respectively).  

\section{Preliminaries} 
\label{s:preliminary}

\subsection{Algebraic groups}
Let $R$ be a ring. 
Then an \emph{$R$-group scheme} is a data of a quadruple $(X,m,\text{inv},e)$  where
$X$ is an $R$-scheme of finite type,  
$m \colon X\times_R X \to X$ and $\text{inv}\colon X \to X$ 
are morphisms of $R$-schemes, and the neutral point $e \in X(R)$ such that they verify formally the axioms of a group (cf. \cite[Definition~1.1]{milne}). 
\par 
When the ring $R$ is a field, we say that $X$ is an algebraic group. Prominent examples of algebraic groups include \emph{linear algebraic groups} and \emph{abelian varieties}. 
\par 
An $R$-scheme morphism $\varphi \colon X \to Y$ of $R$-group schemes is  
a \emph{homomorphism of $R$-group schemes} if $\varphi \circ m_X = m_Y \circ (\varphi \times \varphi)$. 
Here, $m_X$ and $m_Y$ denote  respectively the group laws of $X$ and $Y$. See the references \cite{grothendieck-ega-1-1} and \cite{milne} for more details.

\subsection{Models of homomorphisms of algebraic groups}
\label{s:model-finite-data} 

We shall need the following   standard lemma in algebraic geometry: 
\begin{lemma}
\label{l:model-finite-data}
Let $X, Y$ be algebraic groups over a field $k$. Let $f_i \colon X^{n_i} \to Y^{m_i}$,  $m_i,n_i \in \N$, $i \in I$, be finitely many homomorphisms of $k$-algebraic groups. Then there exist a finitely generated $\Z$-algebra $R \subset k$ and $R$-group schemes of finite type $X_R$, $Y_R$ and $R$-homomorphisms $f_{i,R} \colon (X_R)^{n_i} \to (Y_R)^{m_i}$ of $R$-group schemes with $X=X_R \otimes_R k$, $Y= Y_R \otimes_R k$ such that $f_i=f_{i,R} \otimes_R k$ (base change to $k$). If $X=Y$, one can take $X_R=Y_R$. If $f_{i}$ is a closed immersion the one can also choose $f_{i,R}$ to be a closed immersion. 
\end{lemma}

\begin{proof}
See, e.g., \cite[Section~8.8]{ega-4-3}, notably \cite[Scholie~8.8.3]{ega-4-3}, and  \cite[Proposition~8.9.1]{ega-4-3}. 
\end{proof}

\subsection{Closed points of $\Z$-schemes of finite type} To perform the reduction to the finite alphabet case during the proof of Theorem~\ref{t:surjunctive-invertible}, we shall need the following auxiliary lemma to describe the set of closed points of families of  schemes of finite type over $\Z$.   
\begin{lemma}
\label{l:Z-closed-point} 
Let $X$ be an $S$-scheme of finite type where $S$ is a $\Z$-scheme of finite type. Then for every finite set $M$, the set of closed points of $X_S^M$ is the union of finite sets  $\cup_{s,p,d} H_{s,p,d}^M$ over all prime numbers $p \in \N$, all closed points $s \in S_p=S \otimes_\Z \mathbb{F}_p$, and all $d \in \N$ where 
\begin{equation*}
    H_{p,s, d} = \{ x \in X_s \colon \vert \kappa (x) \vert = p^r, 1 \leq r \leq d\} 
\end{equation*}
where $X_s= X \otimes_S \kappa(s)$ is the fibre of $X$ above $s$.  
\end{lemma}

\begin{proof}
See~\cite[Corollary~3.3]{phung-geometric}
\end{proof}

\subsection{Inverses of injective and bijective endomorphisms}

In Section~\ref{s:proof-surjunctivity}, we will need the following useful preliminary results in \cite{phung-2020}. 

\begin{lemma}
\label{l:techno-left-reversible}
Given a group $G$ and an algebraic group $X$ over an algebraically closed field $k$. 
Suppose that $\tau\in CA_{algr}(G,X,k)$ is injective. 
Denote $A= X(k)$ and $\Gamma= \tau(A^G)$.
Then there exists a finite subset $N\subset G$ such that for every $d\in \Gamma$, 
 $\tau^{-1}(d)(1_G)$ depends only on $d \vert_N \in A^N$. 
\end{lemma} 

\begin{proof}
See \cite[Lemma~6.1]{phung-2020}. 
\end{proof}

\begin{theorem}
\label{t:inverse-also}
Given a group $G$ and an algebraic group $X$ over an algebraically closed field $k$ of characteristic zero. 
Suppose that $\tau \in CA_{algr}(G,X, k)$ is bijective. 
Then one has $\tau^{-1}\in CA_{algr}(G,X,k)$.  
\end{theorem}

\begin{proof}
See \cite[Theorem~6.4]{phung-2020}. 
\end{proof}

\section{Locally weakly surjunctive groups}
\label{s:weakly-surjunctive-groups} 

Given a property $(\mathrm{P})$ of groups, we say that a group $G$ is locally $(\mathrm{P})$ if every finitely subgroup of $G$ has property $(\mathrm{P})$. 
In \cite{weiss-sgds}, the author proved that a group is surjunctive if and only if it is locally surjunctive. 
\par 
The below lemma shows that the similar property also holds for weakly surjunctive groups. 

\begin{lemma}
\label{l:locally-surjunctive}
A group is locally weakly surjunctive if and only if it is weakly surjunctive. 
\end{lemma}

\begin{proof}
\par 
Fix a group $G$ and a finite group $A$. Let $\tau \colon A^G \to A^G$ be an injective group cellular automaton. Let $\mu \colon A^M \to A$ be a local defining map of $\tau$ associated with some finite memory set $M \subset G$. Since $\tau$ is a group cellular automaton, it is clear that $\mu$ is a homomorphism of groups. 
\par 
Let $H \subset G$ be the subgroup generated by $M$. 
Let $\tau_H \colon A^H \to A^H$ be the cellular automaton which admits $\mu \colon A^M \to A$ as a local defining map. Since $\mu$ is a homomorphism of groups, we find that $\tau_H$ is also a group cellular automaton. 
By the main result of \cite{csc-induction},  we know that $\tau$ is injective, resp. surjective, if and only if so is $\tau_H$.  
\par 
Suppose first that $G$ is locally weakly surjunctive. Then $H$ is a weakly surjunctive group. 
Since $\tau$ is injective, so is $\tau_H$.  Therefore, $\tau_H$ is surjective as $H$ is weakly surjunctive. Hence, $\tau$ is also  surjective. This proves that $G$ is weakly surjunctive. 
\par
The other implication is similar. Suppose that $G$ is weakly surjunctive. Let $S\subset G$ be a finitely generated subgroup. Let $\tau_S \colon A^S \to A^S$ be an injective group cellular automaton. Let $\mu \colon A^N \to A$ be a local defining map of $\tau_H$ with a  memory set $N \subset H$. Note that $\mu$ is a group homomorphism. 
\par 
Consider the group cellular automaton $\tau \colon A^G \to A^G$ which admits $\mu$ as a local defining map. Since $\tau_S$ is injective, $\tau$ is also injective. But as $G$ is weakly surjunctive, we deduce that $\tau$ is surjective and thus so is $\tau_S$. This shows that $S$ is weakly surjunctive. The proof is complete.  
\end{proof}

\section{Proof of the main result}
\label{s:proof-main} 

\subsection{Surjunctivity} 
\label{s:proof-surjunctivity}

When the universe is an amenable group, the surjunctivity of endomorphisms of symbolic group varieties is a direct  consequence of the recent result proved in  \cite[Theorem~B]{phung-post-surjective}: 

\begin{theorem} 
 \label{t:gromov-answer-algr-full} 
 Let $G$ be an amenable group and let $X$ be an algebraic group over  an algebraically closed field $k$. Suppose that $\tau \in CA_{algr}(G,X,k)$ is pre-injective (e.g. when $\tau$ is injective). Then $\tau$ is also surjective. \qed
\end{theorem}
\par 
It is known (cf.~\cite[Theorem~1.1]{phung-2020}) that the conclusion of the above theorem also holds if the group $G$ is sofic but $X$ is connected. 
By further requiring the base field $k$ to be uncountable, we can prove the following general surjunctivity result for endomorphisms of symbolic group varieties: 

\begin{theorem}
\label{t:main-surjunctive-1}
Let $G$ be a weakly surjunctive group and let $X$ be an algebraic group over an uncountable algebraically closed field $k$. Suppose that $\tau \in CA_{algr}(G,X,k)$ is injective. Then $\tau$ is also surjective. 
\end{theorem}

The proof of Theorem~\ref{t:main-surjunctive-1} will occupy the rest of Section~\ref{s:proof-surjunctivity}. Fixing the notations and hypotheses as in Theorem~\ref{t:main-surjunctive-1}, we denote by $A= X(k)$ the set of $k$-points of $X$ and $\Gamma= \tau(A^G)\subset A^G$ the image of $\tau$. Let $e \in A$ be the neutral element of $X$. For all sets $E \subset F$ and $\Lambda \subset A^F$,  we denote  $\Lambda_E=\{ x\vert_{E}\colon x \in \Lambda \} \subset A^E$. Since $\tau \colon A^G \to A^G$ concerns only the set of $k$-points $A=X(k)$, we can clearly suppose without loss of generality that $X$ is a reduced scheme.
\par 
Let us fix an arbitrary symmetric memory set $M \subset G$ of $\tau$ such that $1_G \in M$ and let $\mu \colon X^M \to X$ be the homomorphism of algebraic groups which serves as the corresponding local defining map of $\tau$. 
\par 
For every finite subset $E \subset G$, we define a homomorphism of algebraic groups $\tau_E^+ \colon X^{E M } \to X^E$ defined by $\tau^+_E(c)(g)= \mu((g^{-1}c)\vert_M)$ for every $c \in A^{E M}$ and $g \in E$. 
Since $M$ is a memory set of $\tau$,  we have  $\Gamma_E = \tau_E^+ (A^{E M})$.   
\par 
Up to replacing $G$ by the subgroup generated by $M$, we can suppose without loss of generality that $G$ is finitely generated (by using the main result of \cite{csc-induction} as in the proof of Lemma~\ref{l:locally-surjunctive}). In particular, we can find an increasing sequence of finite subsets $(E_n)_{n \in \N}$ of $G$ with $M \subset E_0$
and such that $G = \cup_{n \in \N} E_n$. 
\par 
For each $n \in \N$, we consider the following subset $V_n$ of  $X^{E_n M }$: 
\begin{align}
    \label{e:V_n-proof-1} 
    V_n \coloneqq \{x \in A^{E_n M} \colon x(1_G) \neq e , \tau_{E_n}^+(x)=e^{E_n}\}. 
\end{align}
\par 
Observe that $V_n$ is a constructible subset of $X^{E_nM}$. It suffices to write: 
\begin{equation}
\label{e:main-proof-v-n-1-2}
V_n= \Ker \tau_{E_n}^+ \setminus \Ker \pi_n 
\end{equation}
where $\pi_n \colon X^{E_n M} \to X^{\{1_G\}}$ is the canonical projection $x \mapsto x(1_G)$. 
\par 
For integers $m \geq n \geq 0$, we denote by $p_{m,n} \colon X^{E_m M} \to X^{E_n M} $ the canonical projection map given by $x \mapsto x\vert_{E_n M}$. Then 
by the construction of the maps $\tau_{E_n}^+$, it is clear that we have 
$p_{m,n} (V_m) \subset V_n$ for all $m \geq n \geq 0$. 
\par 
\underline{\textbf{Claim}}: There exists $n \in \N$ such that $V_n = \varnothing$. Indeed, 
suppose on the contrary that $V_n \neq \varnothing$ for every $n \in \N$. Then  $(V_n)_{n \in \N}$ forms a projective system of nonempty constructible subsets of algebraic varieties $(X^{E_nM})_{n \in \N}$ with algebraic transition maps induced by the morphisms $p_{m,n}$ with $m \geq n \geq 0$. 
\par 
Since the base field $k$ is uncountable and algebraically closed, we infer from \cite[Lemma~B.2]{cscp-alg-ca} or \cite[Lemma~3.2]{cscp-invariant-ca-alg} that $\varprojlim_{n} V_n \neq \varnothing$. Therefore, we can choose 
\[
c \in \varprojlim_{n} V_n \subset \varprojlim_{n} A^{E_n M}= A^G.
\]
\par 
By definition of the sets $V_n$, we deduce that $c(1_G) \neq e$ and $\tau(x)=e^G$. In particular, $c \neq e^G$ but $\tau(c)=\tau(e^G)=e^G$, which contradicts the injectivity of $\tau$. The claim is thus proved.
\par 
Now let us fix $n \geq 0$ such that $V_n = \varnothing$. Let $W= \Ker \tau_{E_n}^+$ and $U= \Ker \pi_n$ then it follows from \eqref{e:main-proof-v-n-1-2} that $W \subset U$ is a closed algebraic subgroup. 
\par 
By Lemma~\ref{l:model-finite-data}, we can find a finitely generated $\Z$-algebra $R \subset k$ and $R$-group schemes of finite type $X_R$, $W_R$, and $U_R$, as well as  $R$-homomorphisms of group schemes: 
\[
  \mu_R \colon (X_R)^M \to X_R, \quad 
T_R \colon (X_R)^{E_n M } \to (X_R)^{E_n}
\] 
such that the following hold: 
\begin{enumerate}[\rm (i)] 
    \item 
    $X= X_R \otimes_R k$,   $\mu= \mu_R \otimes_R k$;
    \item  $T_R$ is induced by $\mu_R$ so that $\tau_{E_n}^+ = T_R \otimes_R k$; 
    \item 
    $W_R \subset U_R$ is a closed $R$-group subscheme;  \item 
    $W_R= \Ker T_R$,  $U_R= \Ker \pi_R$; 
    \end{enumerate}
    where $\pi_R \colon (X_R)^{E_n M} \to (X_R)^{\{1_G\}}$ is the canonical projection. 
It follows in particular that $W=W_R \otimes_R k$ and  $U=U_R \otimes_R k$. 
\par 
In the condition (ii), by saying $T_R$ is induced by $\mu_R$, we mean that $T_R$ is defined by the universal proprety of fibered product by the collection of  $R$-homomorphism of group schemes $(T_{E_n,g})_{g \in E_n}$  where the component $R$-homomorphism $T_{E_n,g} \colon (X_R)^{E_n M} \to (X_R)^{\{g\}}$, for $g \in E_n$, is given by the composition of the projection $(X_R)^{E_nM} \to (X_R)^{g M}$ followed by the $R$-homomorphism $(X_R)^{gM} \to (X_R)^{\{g\}}$ induced by $\mu$ via the trivial isomorphism $(X_R)^{gM} \simeq (X_R)^M$ and $(X_R)^{\{G\}}\simeq (X_R)^{\{1_G\}}$ obtained by the reindexing bijection $M \to gM$ given by $h \mapsto gh$. 
\par 
We are now in position to prove Theorem~\ref{t:main-surjunctive-1}. 
\begin{proof}[Proof of Theorem~ \ref{t:main-surjunctive-1}] 
Let us denote $S=\Spec R$ then $S$ is a $\Z$-scheme of finite type. 
Note that by Lemma~\ref{l:Z-closed-point}, the set of closed points of $X_R^{E_n}$ is given by $\Delta = \cup_{p \in \PP, s \in S_p,  d\in \N} H_{p,s,d}^{E_n}$  where 
 $\PP$ is the set of prime numbers, $s \in S_p= S \otimes_\Z \mathbb{F}_p$ is a closed point and 
\begin{equation}
   H_{p,s, d}= \{x \in X_s \colon \vert \kappa(x) \vert=p^r, 1 \leq r \leq d\} 
\end{equation}
is a finite subset of closed points of the fibre $X_s=X_R \otimes_R \kappa(s)$.  
\par 
Now fix $p \in \PP$, a closed point $s \in S_p$, and $d \in \N$. 
Let $\mu_s= \mu_R \otimes _R k$ then $\mu_s \colon (X_s)^M \to X_s$ is a $\kappa(s)$-homomorphism of algebraic groups.  
\par 
Since $\mu_{s}(H_{p,s,d}^M) \subset H_{p,s,d}$ (see e.g. \cite[Lemma~3.1]{phung-geometric}), we obtain a well-defined cellular automaton 
$\tau_{p,s,d}  \colon H_{p,s,d}^G \to H_{p,s,d}^G$ admitting $\mu_{s}\vert_{H_{p,s,d}^M}$ as a local defining map. 
\par 
Observe that $H_{p,s,d}$ is naturally a subgroup of $X_s$ (by~e.g. \cite[Lemma~3.1]{phung-geometric}). Since $\mu_s$ is a homomorphism, so is  $\mu_{s}\vert_{H_{p,s,d}^M}$. It follows that $\tau_{p,s,d}$ is a group cellular automaton. 
\par 
\underline{\textbf{Claim}}: $\tau_{p,s,d}$ is injective.  Indeed, suppose that $y \in T_{p,s,d}^G\setminus \{e_s^G\}$ satisfies  $\tau_{p,s,d}(y)=e_s^G$ where $e_s \in X_s $ is the neutral element. 
\par 
By $G$-equivariance, we can assume that $y(1_G) \neq e_s$. In particular, if we let $x=y\vert_{E_n M}$ and $T_s = T_R \otimes_R \kappa(s)$ then $T_s(x)=e_s^{E_n}$. Hence, $x \in W_s= \Ker T_s$. 
\par 
Since $W_s \subset U_s = \Ker \pi_s$ where $\pi_s = \pi_R \otimes_R \kappa(s)$, we can  deduce that $y(1_G)=x(1_G)= e_s$ which is a contradiction. Therefore, $\tau_s$ is injective and the claim is proved. 
\par 
\underline{\textbf{Claim}}: $T_R$ is surjective. Indeed, since $G$ is weakly surjunctive by hypothesis and as $\tau_{p,s,d}$ is a group cellular automaton with finite group alphabet, $\tau_{p,s,d}$ must be surjective. In particular, we have for every $p \in \PP$,  $d \in \N$, and $s \in S_p$ closed that 
\[
H_{p,s,d}^{E_n} = \left(\tau_{p,s,d}(H_{p,s,d}^G)\right)_{E_n}\subset T_{s}(X_s^{E_n M}) \subset 
T_R((X_R)^{E_n M}). 
\]
\par 
It follows that $T_R((X_R)^{E_n M})$ contains the set $\Delta = \cup_{p \in \PP, s \in S_p,  d\in \N} H_{p,s,d}^{M} $ of closed points of $(X_R)^{E_n}$. Since $(X_R)^{E_n}$ is a Jacobson scheme (cf., e.g. \cite[Section~3]{phung-geometric}), $\Delta$ is dense in every locally closed subset (i.e., every intersection of a closed subset and an open subset) of $(X_R)^{E_n}$. \par 
On the other hand,  Chevalley's theorem \cite[Th\' eor\` eme~1.8.4]{grothendieck-20-1964} implies that  $T_R((X_R)^{E_n M})$ is a constructible subset of $(X_R)^{E_n}$, namely, a finite union of locally closed  subsets. 
We can thus conclude that $T_R((X_R)^{E_n M})= (X_R)^{E_n}$. Therefore, $T_R$ is surjective and the claim is proved. 
\par 
Now as surjectivity is a stable property under base change, the morphism $\tau_{E_n}^+ = T_{R} \otimes_R k$ is also surjective. Hence, $A^{E_n}= \tau_{E_n}^+(A^{E_nM}) = \Gamma_{E_n}$. 
Since $M \subset E_n$ by our choice, we deduce that $\Gamma_M= A^M$. But $M$ can be chosen to be arbitrarily large, we find that $\Gamma_E= A^E$ for every finite subset $ E\subset G$. Thus, $\Gamma$ is dense in $A^G$ with respect to the prodiscrete topology. Since $\Gamma$ is closed in $A^G$ by \cite[Theorem~5.1]{phung-2020}, we conclude that $\Gamma= A^G$ and $\tau$ is therefore surjective. The proof of Theorem~\ref{t:main-surjunctive-1} is complete. 
\end{proof}

\subsection{Reversibility} 
We keep the notations and hypotheses as in Theorem~\ref{t:surjunctive-invertible} and in Section~\ref{s:proof-surjunctivity}.  Then Theorem~\ref{t:main-surjunctive-1} implies that $\Gamma= \tau(A^G)=A^G$. 
We begin with the following lemma: 

\begin{lemma}
\label{l:sur-1}
There exists a finite set $N \subset G$ such that for every finite subset $E \subset G$ containing $N$, we can find a map $\eta_E \colon A^E \to A$   such that for all $x \in A^{E M}$, we have $\eta_E ( \tau_E^+(x))=x(1_G)$.  
\end{lemma}

\begin{proof} 
Let $N\subset G$ be the finite subset given by  Lemma~\ref{l:techno-left-reversible}. Then for every $d \in A^G$, the element $\tau^{-1}(d)(1_G)$ depends only on $d\vert_N \subset A^N$. 
Therefore, for every finite subset $E \subset G$ containing $N$, we have a well-defined map: 
\begin{align*}
    \label{e:sur-1-1} 
    \eta_E \colon A^E \to A,\quad 
    z  \mapsto \tau^{-1}(\tilde{z})(1_G),
\end{align*}
where $\tilde{z} \in A^G$ is any configuration extending $z \in A^E$, i.e., $\tilde{z}\vert_{E}= z$.   
\par
Now, let $x \in A^{E M}$ and $z= \tau_{E}^+(x)$. Choose an arbitrary $\tilde{x} \in A^G$ such that $\tilde{x}\vert_{E M}=x$. It follows that $\tilde{z}= \tau(\tilde{x})$ extends the configuration $z$. Hence, we find without difficulties that:  
\begin{align*}
    \eta_E(\tau_{E}^+(x))= \eta(z) = \tau^{-1}(\tilde{z})(1_G)= \tilde{x}(1_G)=x(1_G).
\end{align*} 
\par 
The proof is thus complete. 
\end{proof}

\par 
We can now prove the first part of the conclusion of Theorem~\ref{t:surjunctive-invertible}. 
\par 

\begin{lemma}
$\tau$ is reversible. 
\end{lemma}

\begin{proof}
By Lemma \ref{l:sur-1} and by $G$-equivariance of the map $\tau$, we find immediately that the inverse map $\tau^{-1} \colon A^G \to A^G$ is a cellular automaton which admits $\eta_N \colon A^N \to A$ as a local defining map.   
Therefore, $\tau$ is reversible. 
\end{proof}

\subsection{Invertibility} 
We continue with the notations and hypotheses as in Theorem~\ref{t:surjunctive-invertible}. We have seen that $\tau$ is bijective by Theorem~\ref{t:main-surjunctive-1}. 
Hence, it suffices to apply Theorem~\ref{t:inverse-also} to see that $\tau^{-1} \in CA_{algr}(G,X,k)$ when the base field $k$ has characteristic zero.
\par The proof of Theorem~\ref{t:surjunctive-invertible} is thus complete.\qed 

\section{Direct finiteness property of $CA_{algr}(G,X,k)$} 
\label{s:proof-direct-finite} 

The goal of this section is to explain two different proofs of Corollary~\ref{t:direct-weak-surjunctive} presented in the Introduction. 
\par 
For the proof, we fix a weakly surjunctive group $G$ and an algebraic group $X$ over an algebraically closed field $k$. Suppose that we are given  $\sigma, \tau \in CA_{algr}(G,X,k)$ with 
 $\sigma \circ \tau= \Id$. We are going to show that $\tau  \circ \sigma = \Id$ as well. 
 
\begin{proof}[First proof of Corollary~\ref{t:direct-weak-surjunctive}] 
Since every homomorphism of algebraic groups is trivially a morphism of algebraic varieties, we find that 
$CA_{algr}(G,X,k) \subset CA_{alg}(G,X,k)$ where $CA_{alg}(G,X,k)$ denotes the monoid of  cellular automata $\pi \colon X(k)^G \to X(k)^G$ which admit for some finite subset $M \subset G$ a local defining map $\rho \colon X(k)^M \to X(k)$ induced by a morphism of algebraic varieties $\chi \colon X^M \to X$, that is, $\rho= \chi \vert_{X(k)^M}$. 
\par 
Consequently, $\tau, \sigma \in \CA_{alg}(G,X,k)$ and satisfy $\sigma \circ \tau= \Id$. Then it suffices to apply the proof and notations of  
\cite[Theorem~C]{phung-geometric} to our  $\tau$ and $\sigma$ and observe that all the morphisms of algebraic varieties and morphisms of $R$-schemes involved in the proof can be taken to be homomorphisms of algebraic groups and homomorphisms of $R$-group schemes by using Lemma~\ref{l:model-finite-data}. 
\par 
Hence, with the notations in the proof of 
\cite[Theorem~C]{phung-geometric}, the induced injective cellular automata $\tau_{p,s,d}\colon T_{p,s,d}^G \to T_{p,s,d}^G$ become injective group cellular automata. But since $G$ is weakly surjunctive, the maps $\tau_{p,s,d}$ are surjective and the rest of the proof of \cite[Theorem~C]{phung-geometric} shows us that $\tau \circ \sigma= \Id$. The proof of Corollary~\ref{t:direct-weak-surjunctive} is thus complete. 
\end{proof}

\begin{proof}[Second proof of Corollary~\ref{t:direct-weak-surjunctive}]
Since we only work with $A=X(k)$, we can suppose without loss of generality that $X$ is reduced so that we can identify $X$ with $X(k)$. Let $M \subset G$ be a large enough  common symmetric memory set of both $\sigma$ and $\tau$ such that $1_G \in M$. Let $\mu, \eta \colon X^M \to X$ be the $k$-morphisms of algebraic varieties which induce the local defining maps of $\tau$ and $\sigma$ respectively. 
\par 
Let us define a $k$-homomorphism of algebraic groups $\tau_M^+ \colon X^{M^2} \to X^M$ defined by $\tau^+_M(c)(g)= \mu((g^{-1}c)\vert_M)$ for $c \in A^{M^2}$ and $g \in M$. Since $\sigma \circ \tau= \Id$, it follows immediately that 
\begin{equation}
    \label{e:direct-finite-proof-2-2}
    \eta \circ \tau_M^+ = \pi 
\end{equation}
where $\pi\colon X^{M^2} \to X^{\{1_G\}}$ is the canonical projection $x \mapsto x(1_G)$. \par 
We embed $k$ into some uncountable algebraically closed field $K$. By making the base change to $K$, we obtain the morphisms $\mu_K= \mu \otimes_k K$ and  $\eta_K= \eta \otimes_k K$ and a $K$-algebraic group $X_K= X\otimes_k K$. It is trivial that $\pi_K= \pi \otimes_k K \colon X_K^{M^2} \to X_K^{\{1_G\}}$ is the canonical projection to the factor $X^{\{1_G\}}$.
\par 
Consider the $K$-homomorphism of algebraic groups  $\tau^+_{K,M}=\tau^+_{M} \otimes_k K$. 
We infer from \eqref{e:direct-finite-proof-2-2} that 
$\eta_K \circ \tau^+_{K,M}= \pi_K$. Denote by $\tau_K, \sigma_K \in \CA_{algr}(G,X_K,K)$ the cellular automata admitting respectively $\mu_K$ and $\eta_K$ as local defining maps. Then the relation $\eta_K \circ \tau^+_{K,M}= \pi_K$ implies that  
$\sigma_K \circ \tau_K= \Id$. 
Therefore, $\tau_K$ is injective and Theorem~\ref{t:surjunctive-invertible} implies that $\tau_K$ is bijective. Consequently, $\tau_K \circ \sigma_K= \Id$.
\par 
Consider the $k$-homomorphism of algebraic groups $\sigma_M^+ \colon X^{M^2} \to X^M$ defined by $\sigma^+_M(c)(g)= \eta((g^{-1}c)\vert_M)$ for $c \in A^{M^2}$ and $g \in M$. 
Let $\sigma_{K,M}^+ \colon X_K^{M^2} \to X_K^M$ be defined by 
$\sigma_{K,M}^+= \sigma_M^+ \otimes_k K$. 
Then the relation $\tau_K \circ \sigma_K= \Id$ implies that $\mu_K \circ \sigma_{K,M}^+=\pi_K$. Since the field $k$ is algebraically closed, it follows from e.g. \cite[Lemma~7.2]{cscp-alg-ca}  that $\mu \circ \sigma_{M}^+=\pi$ and we can finally conclude that $\tau \circ \sigma= \Id$. The proof of Corollary~\ref{t:direct-weak-surjunctive} is thus complete. 
\end{proof}

\section{Linearly surjunctive groups and further applications}
\label{s:generalizations} 
Given a group $G$ and a vector space $A$ over a field $k$. Then $A^G$ is naturally a vector space over $k$ with pointwise operations. A cellular automaton $\tau \colon A^G \to A^G$ is said to be \emph{linear} if it is also a  $k$-linear map of vector spaces. \par 
Equivalently, we see that a cellular automaton is linear if it admits a linear local defining map. 
\par 
It is worth notice that linear cellular automata and linear subshifts are known to  admit many interesting interactions with group theory as well as ring theory (cf. e.g.  \cite{csc-sofic-linear}, \cite{bartholdi}, \cite{bartholdi-kielak}, \cite{phung-israel}, \cite{phung-dcds},  \cite{cscp-jpaa}). 
\par 
From our main results obtained in the above sections, it is natural to define the following general class of linearly surjunctive groups. 
\begin{definition}
We say that a group $G$ is a \emph{linearly surjunctive group} if for every finite dimensional vector space $A$ over a finite field $k$, all injective linear cellular automata $\tau \colon A^G \to A^G$ are  surjective. 
\end{definition}

\par 
Hence, it follows from the definitions and the Gromov-Weiss theorem (\cite{gromov-esav},  \cite{weiss-sgds})  which says that all sofic groups are surjunctive, we have the following relations between various classes of groups:  
\begin{align*}
\{ \text{sofic groups}\} & \subset 
\{ \text{surjunctive groups}\} \\
& \subset \{ \text{weakly surjunctive groups}\} \\
& \subset \{ \text{linearly surjunctive groups}\}.
\end{align*}
\par
It is not hard to see using a similar proof of Lemma~\ref{l:locally-surjunctive} that a group is linearly surjunctive if and only if it is locally linearly surjunctive. 
\par 
By working in the smaller category of finitely generated modules over a finitely generated $\Z$-algebra instead of the category of group schemes, we find that a straightforward modification of the proof of Theorem~\ref{t:surjunctive-invertible} shows the following result: 
\begin{theorem}
\label{t:linear-surjunctive-invertible}
Let $G$ be a linearly surjunctive group and let $A$ be a finite dimensional vector space over an uncountable algebraically closed field $k$. 
Suppose that a linear cellular automaton $\tau\colon A^G \to A^G$ is injective. Then $\tau$ is reversible and  $\tau^{-1}\colon A^G \to A^G$ is also a linear cellular automaton. 
\end{theorem}

\begin{proof}
It suffices to follow, \emph{mutatis mutandis}, the same proof of Theorem~\ref{t:surjunctive-invertible} in Section~\ref{s:proof-main} to see that $\tau$ is reversible. We remark that the various $R$-group schemes and homomorphisms between them are now simply $R$-modules of finite type and homomorphisms of $R$-modules. For the second statement of the conclusion, observe that the inverse cellular automaton $\tau^{-1}\colon A^G\to A^G$ is automatically linear as the inverse of a linear map. 
\end{proof}
\par 
As an immediate consequence, we obtain: 

\begin{theorem}
\label{t:kap-linear-surjunctive} 
Kaplansky's stable finiteness theorem holds for every field $k$ and every linearly surjunctive group $G$. 
\end{theorem}

\begin{proof}
We can clearly embed $k$ into some  uncountable algebraically closed field $K$. Let $n \geq 1$ be an integer and suppose that  $\alpha, \beta \in \mathrm{Mat}_n(k[G])$ satisfies $\beta \alpha= 1$. 
Since $\mathrm{Mat}_n(k[G]) \subset \mathrm{Mat}_n(K[G])$ and we have a ring isomorphism $\Phi \colon \mathrm{Mat}_n(K[G]) \to LCA(G,K^n)$ where $LCA(G,K^n)$ is the $K$-algebra of linear cellular automata $(K^n)^G \to (K^n)^G$ (cf. \cite[Corollary~8.7.8]{ca-and-groups-springer}). It follows that $\Phi(\beta)\circ \Phi(\alpha)= \Phi(\beta \alpha)= \Phi(1)=\Id$. In particular, the linear cellular automaton $\Phi(\alpha)$ is injective. Hence, Theorem~\ref{t:linear-surjunctive-invertible} implies that $\Phi(\alpha)$ is bijective and consequently, $\Phi(\alpha) \circ \Phi(\beta)=\Id$. We deduce that $\Phi(\alpha \beta)= \Id$ and thus $\alpha \beta=1$. The proof is  complete. 
\end{proof}
\par 
In \cite[Definition~1.1]{csc-sofic-linear}, the authors introduced the notion of \emph{$L$-surjunctive groups}: a group $G$ is {$L$-surjunctive} if for every finite dimensional vector space $V$ over a field $k$ (not necessarily finite), all injective linear cellular automata $\tau \colon V^G \to V^G$ are  surjective. They showed in \cite{csc-sofic-linear} that a group $G$ is $L$-surjunctive if and only if for every field $k$, the group ring $k[G]$ is stably finite.
In other words, a group satisfies Kaplansky's stable finiteness conjecture for every field if and only if it is $L$-surjunctive. By definition, observe that every $L$-surjunctive group is linearly surjunctive.  
\par 
As a direct application of Theorem~\ref{t:kap-linear-surjunctive}, we deduce from the above characterization of $L$-surjuntive groups that: 
\begin{theorem}
\label{t:L-surjunctive-linear} 
A group is $L$-surjunctive if and only if it is linearly surjunctive. In particular, a group $G$ is  linearly surjunctive if and only if  for every  field $k$, the group ring $k[G]$ is stably finite. \qed
\end{theorem}

\section{Weakly dual surjunctive groups} 
\label{s:weakly-dual-surjunctive}
Fix a group $G$ and a set $A$. Two configurations $x,y\in A^G$ are said to be \emph{asymptotic} if there exists a finite subset $E \subset G$ such that $x\vert_{G \setminus E}= y_{G \setminus E}$. 
Following \cite{kari-post-surjective}, a cellular automaton $\tau \colon A^ G \to A^G$ is said to be \emph{post-surjective} if for every $x, y \in A^G$ such that $y$ is asymptotic to $\tau(x)$, we can find $z \in A^G$ asymptotic to $x$ such that $\tau(z)=y$.  
It was shown in \cite[Theorem~2]{kari-post-surjective} that for every sofic group $G$ and every finite set $A$, every post-surjective cellular automaton $\tau \colon A^G \to A^G$ is pre-injective. 
\par 
Following \cite{kari-post-surjective}, we will say that a group $G$ is \emph{dual surjunctive} if for every finite alphabet $A$, all post-surjective cellular automaton $\tau \colon A^G \to A^G$ are pre-injective. As for weakly surjunctive groups, we introduce the following notions of weakly and linearly dual surjunctive groups which include all dual surjunctive groups and hence all sofic groups.   
\par
\begin{definition}
\label{d:weakly-dual} 
Let $G$ be a group. We say that $G$ is \emph{weakly dual surjunctive} if for every finite group $A$, every post-surjective group cellular automaton $ A^G \to A^G$ is pre-injective. 
Similarly, $G$ is \emph{linearly dual surjunctive} if for every finite dimensional vector space $A$ over a finite field, all linear post-surjective cellular automata $A^G \to A^G$ must be pre-injective. 
\end{definition} 
\par 
Note that by \cite[Theorem~1]{kari-post-surjective}, every pre-injective, post-surjective cellular automaton over a finite alphabet must be bijective and moreover reversible. 
\par 
We begin with the following basic observation: 
\begin{proposition}
\label{p:bartholdi-linearly-surjunctive-dual}
Let $G$ be a group. Then $G$ is linearly dual surjunctive if and only if it is linearly surjunctive.
\end{proposition}

\begin{proof}
It is a direct consequence of \cite[Theorem~1.4]{bartholdi}. 
\end{proof}
\par 
Combining Proposition~\ref{p:bartholdi-linearly-surjunctive-dual} and  Theorem~\ref{t:L-surjunctive-linear}, we obtain immediately the following characterizations of Kaplansky's stable finiteness conjecture: 

\begin{corollary}
Let $G$ be a group. Then the following are equivalent:
\begin{enumerate}[\rm (i)] 
    \item $G$ is $L$-surjunctive; 
    \item $G$ is linearly surjunctive; 
    \item $G$ is linearly dual surjunctive; 
    \item $G$ satisfies Kaplansky's stable finiteness conjecture, that is, for every field $k$, the group ring $k[G]$ is stably finite. \qed
\end{enumerate}
\end{corollary}
\par 
We also obtain the following direct finiteness property of endomorphisms of symbolic group varieties over weakly dual surjunctive group universes:  

\begin{theorem}
\label{t:stable-weakly-dual} 
Let $G$ be a weakly dual surjunctive group. Let $X$ be an algebraic group over an algebraically closed field $K$. Suppose that $\tau, \sigma \in CA_{algr}(G,X,K)$ satisfy $\sigma \circ \tau= \Id$. Then one has $\tau \circ \sigma= \Id$. 
\end{theorem}

As a direct application, we obtain:  
\begin{corollary}
\label{c:stable-weakly-dual} 
Let $G$ be a weakly dual surjunctive group and let $R$ be the endomorphism ring of a commutative algebraic group over an algebraically closed field. Then the group ring $R[G]$ is stably finite. 
\end{corollary}

\begin{proof}
It suffices to follow the exact same proof of Corollary~\ref{c:intro-1}. The only modification needed in the proof is that instead of using  Corollary~\ref{t:direct-weak-surjunctive}, we  apply Theorem~\ref{t:stable-weakly-dual}. 
\end{proof}
\par 
For the proof of Theorem~\ref{t:stable-weakly-dual}, we shall  establish first the following elementary result. 
\begin{lemma}
\label{l:weakly-dual-surjunctive-direct-finite}
Let $G$ and $A$ be groups and let $\tau, \sigma \colon A^G\to A^G$ be group cellular automata such that $\sigma \circ \tau= \Id$. Then $\sigma$ is post-surjective. 
\end{lemma}

\begin{proof}
Suppose that $x,y \in A^G$ with $y$ asymptotic to $\sigma(x)$. Since $\tau$ is uniformly continuous as a cellular automaton, we deduce that $\tau(y)$ and $\tau(\sigma(x))$ are also asymptotic. Consequently, the configuration $z=\tau(\sigma(x)^{-1}y)$ is asymptotic to $e^G$ where $e$ is the neutral element of the group $A$. It follows in particular that $xz$ is aysmptotic to $x$. 
Since $\tau$ and $\sigma$ are group homomorphisms and $\sigma \circ \tau= \Id$, we find that 
\begin{align*}
    \sigma(xz)= \sigma(x) \sigma(z)= \sigma(x) \sigma(\tau(\sigma(x)^{-1}y)) = 
    \sigma (x) \sigma(x)^{-1}y= y.
\end{align*}
\par 
The proof is thus complete. 
\end{proof}
\par 
We are now in the position to give the proof of Theorem~\ref{t:stable-weakly-dual}. 

\begin{proof}[Proof of Theorem~\ref{t:stable-weakly-dual}] 
As for the first proof of  Corollary~\ref{t:direct-weak-surjunctive}, we proceed by applying the constructions and notations introduced in the proof of 
\cite[Theorem~C]{phung-geometric} to our  $\tau$ and $\sigma$. Again, observe that the various  morphisms of algebraic varieties and morphisms of $R$-schemes involved are now respectively  homomorphisms of algebraic groups and homomorphisms of $R$-group schemes by using Lemma~\ref{l:model-finite-data}.
\par 
Let $p$ be a prime, let $s \in A_s$ be a closed point, and let $d \in \N$. Then as for $\tau_{p,s,d} \colon T_{p,s,d}^G \to T_{p,s,d}^G$, we also have a well-defined group cellular automaton 
$\sigma_{p,s,d}=\sigma_{s} \vert_{T_{p,s,d}^G} \colon T_{p,s,d}^G \to T_{p,s,d}^G$ which admits  $\eta_{s}\vert_{T_{p,s,d}^M} \colon T_{p,s,d}^M \to T_{p,s,d}$ as a local defining map. 
\par 
Since $\sigma_s \circ \tau_s=\Id$, we deduce that  $\sigma_{p,s,d} \circ \tau_{p,s,d}=\Id$. We infer from Lemma~\ref{l:weakly-dual-surjunctive-direct-finite} that $\sigma_{p,s,d}$ is a post-surjective group cellular automata. But as the group $G$ is weakly dual surjunctive by hypothesis,  $\sigma_{p,s,d}$ and $\tau_{p,s,d}$ are mutual inverses by \cite[Theorem~1]{kari-post-surjective}. It follows that $\tau_{p,s,d}\circ \sigma_{p,s,d}=\Id$. Therefore, we deduce immediately that $\tau$ is surjective and again, the rest of the proof of \cite[Theorem~C]{phung-geometric} allows us to conclude that $\tau \circ \sigma= \Id$.
\par 
The proof of Theorem~\ref{t:stable-weakly-dual} is thus complete. 
\end{proof}

\bibliographystyle{siam}

\end{document}